\documentclass[12pt]{article}

\usepackage{amssymb}
\usepackage{latexsym}
\usepackage{amsmath}
\usepackage{indentfirst}
\usepackage{color}
\usepackage{thmtools}
\usepackage{thm-restate}
\usepackage{hyperref}
\usepackage{ulem}

\title{Star-critical Ramsey numbers and regular Ramsey numbers for stars}
\author{Zhidan Luo{\thanks{School of Mathematics and Statistics, Hainan University, Haikou 570228, P. R. China. Email: luodan@hainanu.edu.cn.}}}
\date{}

\newtheorem{theo}{Theorem}[section]

\newtheorem{coro}[theo]{Corollary}
\newtheorem{conj}[theo]{Conjecture}

\newtheorem{fact}[theo]{Fact}

\def\q{\hspace*{\fill}$\Box$\medskip}

\begin{document}

  \maketitle

  \begin{abstract}
    Let $G$ be a graph, $H$ be a subgraph of $G$, and let $G- H$ be the graph obtained from $G$ by removing a copy of $H$. Let $K_{1, n}$ be the star on $n+ 1$ vertices. Let $t\geq 2$ be an integer and $H_{1}, \dots, H_{t}$ and $H$ be graphs, and let $H\rightarrow (H_{1}, \dots, H_{t})$ denote that every $t$ coloring of $E(H)$ yields a monochromatic copy of $H_{i}$ in color $i$ for some $i\in [t]$. Ramsey number $r(H_{1}, \dots, H_{t})$ is the minimum integer $N$ such that $K_{N}\rightarrow (H_{1}, \dots, H_{t})$. Star-critical Ramsey number $r_{*}(H_{1}, \dots, H_{t})$ is the minimum integer $k$ such that $K_{N}- K_{1, N- 1- k}\rightarrow (H_{1}, \dots, H_{t})$ where $N= r(H_{1}, \dots, H_{t})$. Let $rr(H_{1}, \dots, H_{t})$ be the regular Ramsey number for $H_{1}, \dots, H_{t}$, which is the minimum integer $r$ such that if $G$ is an $r$-regular graph on $r(H_{1}, \dots, H_{t})$ vertices, then $G\rightarrow (H_{1}, \dots, H_{t})$. Let $m_{1}, \dots, m_{t}$ be integers larger than one, exactly $k$ of which are even. In this paper, we prove that if $k\geq 2$ is even, then $r_{*}(K_{1, m_{1}}, \dots, K_{1, m_{t}})= \sum_{i= 1}^{t} m_{i}- t+ 1- \frac{k}{2}$ which disproves a conjecture of Budden and DeJonge in 2022. Furthermore, we prove that
    $$rr(K_{1, m_{1}}, \dots, K_{1, m_{t}})= \begin{cases}
      \sum_{i= 1}^{t} m_{i}- t, & \text{$k\geq 2$ is even},\\
      \sum_{i= 1}^{t} m_{i}- t+ 1, & otherwise.
    \end{cases}$$

  \end{abstract}

\section{Introduction}

  Let $V(G)$ and $E(G)$ be the vertex set and the edge set of $G$, respectively. Let $K_{1, n}$ be the star on $n+ 1$ vertices. Let $t\geq 2$ be an integer and $H, H_{1}, \dots, H_{t}$ be graphs, and let $H\rightarrow (H_{1}, \dots, H_{t})$ denote that every $t$ coloring of $E(H)$ yields a monochromatic copy of $H_{i}$ in color $i$ for some $i\in [t]$. Ramsey number $r(H_{1}, \dots, H_{t})$ is the minimum integer $N$ such that $K_{N}\rightarrow (H_{1}, \dots, H_{t})$. In 1972, Harary \cite{H1} determined the value of $r(K_{1, n}, K_{1, m})$. And then Burr and Roberts extended it.

  \begin{theo}[Burr and Roberts \cite{BR}]\label{theo1.1}
    If $m_{1}, \dots, m_{t}$ are integers larger than one, exactly $k$ of which are even, then
    $$r(K_{1, m_{1}}, \dots, K_{1, m_{t}})= \begin{cases}
      \sum_{i= 1}^{t} m_{i}- t+ 1, & \text{$k\geq 2$ is even},\\
      \sum_{i= 1}^{t} m_{i}- t+ 2, & otherwise.
    \end{cases}$$
  \end{theo}

  Let $G$ be a graph, $H$ be a subgraph of $G$, and $G- H$ be the graph obtained from $G$ by removing a copy of $H$, i.e., $V(G- H)= V(G)$ and $E(G- H)= E(G)- E(H)$. In 2011, Hook and Isaak \cite{HI} introduced the star-critical Ramsey number $r_{*}(H_{1}, \dots, H_{t})$ which is the minimum integer $k$ such that $K_{N}- K_{1, N- 1- k}\rightarrow (H_{1}, \dots, H_{t})$ where $N= r(H_{1}, \dots, H_{t})$. In 2022, Budden and DeJonge considered the star-critical Ramsey number for stars and conjectured the following.

  \begin{conj}[Budden and DeJonge \cite{BD}]\label{conj1.2}
    If $m_{1}, \dots, m_{t}$ are integers larger than one, exactly $k$ of which are even, then
    $$r_{*}(K_{1, m_{1}}, \dots, K_{1, m_{t}})= \begin{cases}
      \sum_{i= 1}^{t} m_{i}- t, & \text{$k\geq 2$ is even},\\
      1, & otherwise.
    \end{cases}$$
  \end{conj}
  They proved their conjecture for all cases except that $k> 2$ is even.

  For a positive integer $r$, call a graph $r$-regular graph if every vertex has degree $r$. Let $rr(H_{1}, \dots, H_{t})$ be the regular Ramsey number for $H_{1}, \dots, H_{t}$, which is the minimum integer $r$ such that if $G$ is an $r$-regular graph on $N= r(H_{1}, \dots, H_{t})$ vertices, then $G\rightarrow (H_{1}, \dots, H_{t})$. The following holds by the definition of Ramsey number, star-critical Ramsey number, and regular Ramsey number.

  \begin{fact}\label{fact1.3}
    $1\leq r_{*}(H_{1}, \dots, H_{t})\leq rr(H_{1}, \dots, H_{t})\leq r(H_{1}, \dots, H_{t})- 1$.
  \end{fact}

  In this paper, we first disprove Conjecture \ref{conj1.2} for the remaining cases by proving the following.

  \begin{restatable}{theo}{r}\label{theo1.4}
    Let $m_{1}, \dots, m_{k}$ be even integers and $m_{k+ 1}, \dots, m_{t}$ be odd integers larger than one. If $k\geq 2$ is even, then
    $$r_{*}(K_{1, m_{1}}, \dots, K_{1, m_{t}})= \sum_{i= 1}^{t} m_{i}- t+ 1- \frac{k}{2}.$$
  \end{restatable}

  \noindent Then we consider the regular Ramsey number for stars.

  \begin{theo}\label{theo1.5}
    Let $m_{1}, \dots, m_{t}$ be integers larger than one, exactly $k$ of which are even. Then
    $$rr(K_{1, m_{1}}, \dots, K_{1, m_{t}})= \begin{cases}
      \sum_{i= 1}^{t} m_{i}- t, & \text{$k\geq 2$ is even},\\
      \sum_{i= 1}^{t} m_{i}- t+ 1, & otherwise.
    \end{cases}$$
  \end{theo}

  {\bf Notations and definitions}: Let $G\cup H$ be the union of $G$ and $H$, i.e., $V(G\cup H)= V(G)\cup V(H)$ and $E(G\cup H)= E(G)\cup E(H)$. A matching of $G$ is a $1$-regular subgraph of $G$, and a maximum matching of $G$ is a matching of $G$ with the maximum size. For a positive integer $f$, a $f$-factor of graph $G$ is a $f$-regular subgraph of $G$ on $V(G)$.

\section{Star-critical Ramsey number for stars}
  We first introduce a decomposition of a complete graph by Harary in 1969.

  \begin{theo}[Harary \cite{H}]\label{theo2.1}
    $K_{2n}$ can be decomposed into $(2n- 1)$ edge-disjoint $1$-factors, and $K_{2n+ 1}$ can be decomposed into $n$ edge-disjoint $2$-factors.
  \end{theo}

  \noindent Bollab{\'a}s proved a stronger result for the complete graph on odd vertices.

  \begin{theo}[Bollab{\'a}s\cite{B}]\label{theo2.2}
    $K_{2n+ 1}$ can be decomposed into $n$ edge-disjoint Hamiltonian cycles.
  \end{theo}

  \begin{coro}\label{coro2.2}
    Let $n$ be a positive integer. If $n$ is even, then for all $r\in [n- 1]$, there is an $r$-regular graph on $n$ vertices. If $n$ is odd, then for all even $r\in [n- 1]$, there is an $r$-regular graph on $n$ vertices.
  \end{coro}

  \begin{coro}\label{coro2.3}
    Let $n$ be a positive integer, $s\leq n- 1$ be a positive integer, and $G$ be a graph on $n$ vertices without $K_{1, s}$. Then the following holds. If $n$ is odd and $s$ is even, then $e(G)\leq \frac{1}{2}\left[(s- 1)n- 1\right]$. Otherwise, $e(G)\leq \frac{1}{2}(s- 1)n$. Furthermore, the upper bound is the best.
  \end{coro}
  \begin{proof}
     Since $G$ is $K_{1, s}$-free, every vertex has degree at most $s- 1$. And thus, $e(G)\leq \frac{1}{2}(s- 1)n$. If $n$ is odd and $s$ is even, then there is no $(s- 1)$-regular graph on $n$ vertices since the sum of the degree of each graph is even. By Corollary \ref{coro2.2}, there exists a $(s- 2)$-regular graph on $n$ vertices. Thus, at most $n- 1$ vertices have degree $s- 1$ and at least one vertex has degree $s- 2$. Consequently, $e(G)\leq \frac{1}{2}[(s- 1)(n- 1)+ s- 2]$.

     If $n$ is odd and $s$ is even, then by Theorem \ref{theo2.2}, let $C$ be a Hamiltonian cycle on $n$ vertices and $H'$ be a $(s- 2)$-regular graph on $n$ vertices such that $C$ and $H'$ are edge-disjoint. Let $C'$ be a maximum matching of $C$ and let $H= H'\cup C'$. Note that $H$ is a graph on $n$ vertices containing $n- 1$ vertices with degree $s- 1$ and one vertex with degree $s- 2$. Thus, $e(H)= \frac{1}{2}[(s- 1)(n- 1)+ s- 2]$.

     If either $n$ is even or $s$ is odd, then by Corollary \ref{coro2.2}, there exists a $(s- 1)$-regular graph $H$ on $n$ vertices. Thus, $e(H)= \frac{1}{2}(s- 1)n$.\q
  \end{proof}

  Now, we are ready to prove our first result.

  \r*

  \begin{proof}
    Let $N= \sum_{i= 1}^{t} m_{i}- t+ 1$ and $r_{*}= r_{*}(K_{1, m_{1}}, \dots, K_{1, m_{t}})$. Let $V= V(K_{N- 1})$ and $v$ be a vertex. Let $H$ be the graph obtained by joining $v$ and $r_{*}$ vertices of $V$. Color $E(H)$ with $t$ colors arbitrarily. Let $H_{i}$ be the graph induced by all edges with color $i$ in $H$ for every $i\in [t]$.

    Note that $N$ is odd. If $H_{i}$ does not contain $K_{1, m_{i}}$ for every $i\in [t]$, then by Corollary \ref{coro2.3},
    $$e(H)= \sum_{i= 1}^{k} e(G_{i})+ \sum_{i= k+ 1}^{t} e(G_{i})\leq \sum_{i= 1}^{k} \frac{1}{2}[(m_{i}- 1)N- 1]+ \sum_{i= k+ 1}^{t} \frac{1}{2}(m_{i}- 1)N= \frac{1}{2}N(N- 1)- \frac{k}{2}.$$
    Thus, if $e(H)\geq \frac{1}{2}N(N- 1)- \frac{k}{2}+ 1$, then by the pigeonhole principle, there is either $i_{0}\in [k]$ such that $e(H_{i_{0}})\geq \frac{1}{2}[(m_{i}- 1)N+ 1]$ or $i_{0}\in [t]\backslash [k]$ such that $e(H_{i_{0}})\geq \frac{1}{2}(m_{i}- 1)N+ 1$. By Corollary \ref{coro2.3} again, $H_{i_{0}}$ contains a copy of $K_{1, m_{i_{0}}}$. Consequently,
    $$r_{*}\leq \frac{1}{2}N(N- 1)- \frac{k}{2}+ 1- {N- 1\choose 2}= N- \frac{k}{2}$$

    Let $G$ be the graph obtained by joining $v$ and $N- \frac{k}{2}- 1$ vertices of $V$, which will be chosen later. Let $G_{i}$ be the graph induced by all edges with color $i$ in $G$. We will color $E(G)$ with $t$ colors such that $G_{i}$ does not contain $K_{1, m_{i}}$ for every $i\in [t]$.

    By Theorem \ref{theo2.2}, $K_{N}$ can be decomposed into $\frac{N- 1}{2}$ edge-disjoint Hamiltonian cycles and denote them by $C_{i, j}$ where $i\in [t]$ and $j$ satisfies the following: if $i\in [\frac{k}{2}]$, then $j\in [\frac{m_{i}}{2}]$; if $i\in [k]\backslash [\frac{k}{2}]$, then $j\in [\frac{m_{i}}{2}]\backslash [1]$; if $i\in [t]\backslash [k]$, then $j\in [\frac{m_{i}- 1}{2}]$. For every $i\in [\frac{k}{2}]$, let $u_{i}v\in E(C_{i, 1})$ and let $P_{i, 1}$ be the graph obtained from $C_{i, 1}$ by removing $u_{i}v$. Since $C_{i, 1}$ is a Hamiltonian cycle, $P_{i, 1}$ is a Hamiltonian path. Furthermore, since $N$ is odd, there are two edge-disjoint maximum matchings in $P_{i, 1}$ covering all but one vertex. Denote them by $M_{i}$ and $M_{i+ \frac{k}{2}}$, and we may assume that $v\in V(M_{i})$ and $u_{i}\in V(M_{i+ \frac{k}{2}})$. Finally, for every $i\in [k]$, let $G_{i}= M_{i} \cup_{j= 2}^{m_{i}/ 2} C_{i, j}$, and for every $i\in [t]\backslash [k]$, let $G_{i}= \cup_{j= 1}^{(m_{i}- 1)/ 2} C_{i, j}$.

    Note that for every $i\in \left[\frac{k}{2}\right]$, all vertices of $G_{i}$ have degree $m_{i}- 1$ except that $u_{i}$ has degree $m_{i}- 2$; for every $i\in [k]\backslash \left[\frac{k}{2}\right]$, all vertices of $G_{i}$ have degree $m_{i}- 1$ except that $v$ has degree $m_{i}- 2$; for every $i\in [t]\backslash [k]$, all vertices of $G_{i}$ have degree $m_{i}- 1$. Thus, for every $i\in [t]$, $G_{i}$ does not contain $K_{1, m_{i}}$.

    Note that $d_{G}(v)= N- 1- \frac{k}{2}$ and $G[V]= K_{N- 1}$. Thus, $r_{*}\geq N- \frac{k}{2}$ and we finish the proof.\q
  \end{proof}

\section{Regular Ramsey number for stars}

  In this section, we will prove a more general result, and Theorem \ref{theo1.5} is a direct corollary.

  \begin{theo}\label{theo2.8}
    Let $m_{1}, \dots, m_{k}$ be even integers, $m_{k+ 1}, \dots, m_{t}$ be odd integers larger than one, and $n\geq r(K_{1, m_{1}}, \dots, K_{1, m_{t}})$ be an integer. Let $g(n)$ be the minimum integer such that if $G$ is a $g(n)$-regular graph on $n$ vertices, then $G\rightarrow (K_{1, m_{1}}, \dots, K_{1, m_{t}})$. If $n$ is odd and $k\geq 2$ is even, then $g(n)= \sum_{i= 1}^{t} m_{i}- t$. Otherwise, $g(n)= \sum_{i= 1}^{t} m_{i}- t+ 1$.
  \end{theo}
  \begin{proof}
    Note that if $\Delta(G)\geq \sum_{i= 1}^{t} m_{i}- t+ 1$, then by the pigeonhole principle, there is a monochromatic copy of $K_{1, m_{i}}$ in color $i$ for some $i\in [t]$. Thus, $g(n)\leq \sum_{i= 1}^{t} m_{i}- t+ 1$.

    If $n$ is even or $n$ is odd and $k= 0$, then by Corollary \ref{coro2.2} and Theorem \ref{theo2.1}, for every $i\in [t]$, there exists a $(m_{i}- 1)$-regular graph $H_{i}$ on $n$ vertices such that they are edge-disjoint. Note that $\cup_{i= 1}^{t} H_{i}$ is a $\left(\sum_{i= 1}^{t} m_{i}- t\right)$-regular graph such that $H_{i}$ does not contain $K_{1, m_{i}}$ for every $i\in [t]$. Thus, $g(n)\geq \sum_{i= 1}^{t} m_{i}- t+ 1$.

    In the following, assume that $n$ is odd and $k> 0$.

    Case 1: $k$ is odd.

    By Theorem \ref{theo2.2}, $K_{n}$ can be decomposed into $\frac{n- 1}{2}$ edge-disjoint Hamiltonian cycles and denote them by $C_{i, j}$ where $i\in [t]$ and $j$ satisfies the following: if $i\in [\frac{k- 1}{2}]$, then $j\in [\frac{m_{i}}{2}]$; if $i\in [k]\backslash [\frac{k- 1}{2}]$, then $j\in [\frac{m_{i}}{2}]\backslash [1]$; if $i\in [t]\backslash [k]$, then $j\in [\frac{m_{i}- 1}{2}]$. For every $i\in [\frac{k- 1}{2}]$, let $P_{i, 1}$ be the graph obtained from $C_{i, 1}$ by removing an edge $u_{i}u_{k- i}$ such that $\left|\cup_{l= 1}^{(k- 1)/2} \{u_{l}, u_{k- l}\}\right|= k- 1$. Since $C_{i, 1}$ is a Hamiltonian cycle, $P_{i, 1}$ is a Hamiltonian path. Furthermore, since $n$ is odd, there are two edge-disjoint maximum matchings in $P_{i, 1}$ covering all but one vertex. Denote them by $M_{i}$ and $M_{k- i}$, and we may assume that $u_{k- i}\in V(M_{i})$ and $u_{i}\in V(M_{k- i})$. Finally, for every $i\in [k- 1]$, let $G_{i}= M_{i} \cup_{j= 2}^{m_{i}/2} C_{i, j}$, $G_{k}= \left(\cup_{l= 1}^{\frac{k- 1}{2}} u_{l}u_{k- l}\right)\cup_{j= 2}^{m_{k}/2} C_{k, j}$, and for every $i\in [t]\backslash [k]$, let $G_{i}= \cup_{j= 1}^{(m_{i}- 1)/2} C_{i, j}$.

    Note that for every $i\in [k- 1]$, all vertices of $G_{i}$ have degree $m_{i}- 1$ except that $u_{i}$ has degree $m_{i}- 2$; all vertices of $G_{k}$ have degree $m_{k}- 2$ except that $u_{j}$ has degree $m_{k}- 1$ for every $j\in [k- 1]$; for every $i\in [t]\backslash [k]$, all vertices of $G_{i}$ have degree $m_{i}- 1$. Thus, for every $i\in [t]$, $G_{i}$ does not contain $K_{1, m_{i}}$.

    Consequently, $\cup_{i= 1}^{t} G_{i}$ is a $(\sum_{i= 1}^{t} m_{i}- t- 1)$-regular graph on $n$ vertices such that $\cup_{i= 1}^{t} G_{i}\not\rightarrow (K_{1, m_{1}}, \dots, K_{1, m_{t}})$. Note that $\sum_{i= 1}^{t} m_{i}- t$ is odd, and by Corollary \ref{coro2.2}, there does not exist a $(\sum_{i= 1}^{t} m_{i}- t)$-regular graph on $n$ vertices since $n$ is odd. Consequently, $g(n)\geq \sum_{i= 1}^{t} m_{i}- t+ 1$.

    Case 2: $k$ is even.

    Firstly, we improve the upper bound. Otherwise, note that $\sum_{i= 1}^{t} m_{i}- t$ is even, and by Corollary \ref{coro2.2}, there exists a $(\sum_{i= 1}^{t} m_{i}- t)$-regular graph $H$ on $n$ vertices such that $H\not\rightarrow (K_{1, m_{1}}, \dots, K_{1, m_{t}})$. Note that $m_{1}$ is even. By Corollary \ref{coro2.2} again and $n$ is odd, there exists a vertex $u\in V(H)$ such that there are at most $m_{1}- 2$ edges adjacent to $u$ in color $1$. Thus, at least $\sum_{i= 1}^{t- 1} m_{i}- t+ 2$ edges are adjacent to $u$ in the remaining $t- 1$ colors. By the pigeonhole principle, there exists $i_{0}\in [t]\backslash [1]$ such that at least $m_{i_{0}}$ edges are adjacent to $u$ in color $i_{0}$. A contradiction to $H\not\rightarrow (K_{1, m_{1}}, \dots, K_{1, m_{t}})$. Consequently, $g(n)\leq \sum_{i= 1}^{t} m_{i}- t$.

    In the following, we will prove that the equality holds. By Theorem \ref{theo2.2}, $K_{n}$ can be decomposed into $\frac{n- 1}{2}$ edge-disjoint Hamiltonian cycles and denote them by $C_{i, j}$ where $i\in [k]$ and $j$ satisfies the following: if $i\in [\frac{k}{2}]$, then $j\in [\frac{m_{i}}{2}]$; if $i\in [k- 1]\backslash [\frac{k}{2}]$, then $j\in [\frac{m_{i}}{2}]\backslash [1]$; If $i= k$, then $j\in [\frac{m_{k}}{2}]\backslash [2]$; if $i\in [t]\backslash [k]$, then $j\in [\frac{m_{i}- 1}{2}]$. For every $i\in [\frac{k}{2}]$, let $P_{i, 1}$ be the graph obtained from $C_{i, 1}$ by removing an edge $u_{i}u_{i+ \frac{k}{2}}$ such that $\left|\cup_{l= 1}^{k/2} \{u_{l}, u_{l+ \frac{k}{2}}\}\right|= k$. Since $C_{i, 1}$ is a Hamiltonian cycle, $P_{i, 1}$ is a Hamiltonian path. Furthermore, since $n$ is odd, there are two maximum matchings in $P_{i, 1}$ covering all but one vertex. Denote them by $M_{i}$ and $M_{i+ \frac{k}{2}}$, and we may assume that $u_{i+ \frac{k}{2}}\in V(M_{i})$ and $u_{i}\in V(M_{i+ \frac{k}{2}})$. Finally, for every $i\in [k- 1]$, let $G_{i}= M_{i} \cup_{j= 2}^{m_{i}/2} C_{i, j}$, $G_{k}= \left(\cup_{l= 1}^{k/2} u_{l}u_{l+ \frac{k}{2}}\right)\cup M_{k}\cup_{j= 3}^{m_{k}/2} C_{k, j}$, and for every $i\in [t]\backslash [k]$, let $G_{i}= \cup_{j= 1}^{(m_{i}- 1)/2} C_{i, j}$.

    Note that for every $i\in [k- 1]$, all vertices of $G_{i}$ have degree $m_{i}- 1$ except that $u_{i}$ has degree $m_{i}- 2$; all vertices of $G_{k}$ have degree $m_{k}- 3$ except that $u_{j}$ has degree $m_{k}- 2$ for every $j\in [k- 1]$; for every $i\in [t]\backslash [k]$, all vertices of $G_{i}$ have degree $m_{i}- 1$. Thus, for every $i\in [t]$, $G_{i}$ does not contain $K_{1, m_{i}}$.

    Consequently, $\cup_{i= 1}^{t} G_{i}$ is a $(\sum_{i= 1}^{t} m_{i}- t- 2)$-regular graph on $n$ vertices such that $\cup_{i= 1}^{t} G_{i}\not\rightarrow (K_{1, m_{1}}, \dots, K_{1, m_{t}})$. Note that $\sum_{i= 1}^{t} m_{i}- t- 1$ is odd, and by Corollary \ref{coro2.2}, there does not exist a $(\sum_{i= 1}^{t} m_{i}- t- 1)$-regular graph on $n$ vertices since $n$ is odd. Consequently, $g(n)\geq \sum_{i= 1}^{t} m_{i}- t$.

    All cases have been discussed and we finish the proof.\q
  \end{proof}

\section{Remark}

  In \cite{S}, Schelp asked that for graphs $H_{1}, \dots, H_{t}$, if $G$ is a graph on $n$ vertices with $\delta(G)\geq cn$ such that $G\rightarrow (H_{1}, \dots, H_{t})$, then how large should $c$ be? Let $f(n)$ be the minimum integer such that if $G$ is a graph on $n$ vertices with $\delta(G)\geq f(n)$, then $G\rightarrow (H_{1}, \dots, H_{t})$.

  \begin{theo}\label{theo2.9}
    Let $m_{1}, \dots, m_{k}$ be even integers, and $m_{k+ 1}, \dots, m_{t}$ be odd integers larger than one. If $n\geq r(K_{1, m_{1}}, K_{1, m_{2}}, \dots, K_{1, m_{t}})$, then
    $$f(n)= \begin{cases}
      \sum_{i= 1}^{t} m_{i}- t+ 1, & \text{$k= 0$ or $n$ is even},\\
      \sum_{i= 1}^{t} m_{i}- t, & otherwise.
    \end{cases}$$
  \end{theo}
  The proof of Theorem \ref{theo2.9} is the same as the proof of Theorem \ref{theo2.8}, and we remove it.

  The motivation to study regular Ramsey numbers is the following. Let us think of edges in a graph as the resources we need. A complete graph (Ramsey number) is easy to construct but needs many resources. A graph with minimum degree condition (Schelp's problem) needs fewer resources but difficult to construct. And a regular graph (regular Ramsey number) is easier to construct than a graph with minimum degree condition since we can construct by Theorem \ref{theo2.1} and Theorem \ref{theo2.2}, and the resource is less than a complete graph. Thus, the regular Ramsey number has a vast potential for applicants, such as coding or more. Under the limitation of knowledge, we are not able to apply it.

\end{document}